\documentclass[10pt]{amsart}
\usepackage{stmaryrd}
\usepackage{mathrsfs}
\usepackage{amsfonts}
\usepackage{amsmath}
\usepackage{amssymb}
\usepackage{latexsym}
\usepackage{graphicx}
\usepackage{subfigure}
\usepackage{indentfirst}
\usepackage{overpic}
\usepackage{enumerate}
\usepackage{caption2}
\usepackage{cite}
\usepackage{bbding}
\usepackage[active]{srcltx}
\usepackage{cases}
\usepackage{algorithm}
\usepackage{multicol}
\usepackage{bm}
\usepackage{tabularx}
\usepackage{multirow}
\usepackage{booktabs}
\usepackage{color,xcolor,tikz}
\newtheorem{theorem}{Theorem}[section]
\newtheorem{lemma}[theorem]{Lemma}
\newtheorem{corollary}[theorem]{Corollary}

\newtheorem{remark}[theorem]{Remark}

\newcommand{\dx}{\,{\rm d}x}
\newcommand{\dd}{\,{\rm d}}

\begin{document}
\title{Some Error Analysis on Virtual Element Methods}

\author{Long Chen}%
\address{Department of Mathematics, University of California at Irvine, Irvine, CA 92697, USA}%
\email{chenlong@math.uci.edu}%

\author{Jianguo Huang$^{\ast}$}%
\thanks{$^{\ast}$Corresponding author.}
\address{School of Mathematical Sciences, and MOE-LSC, Shanghai Jiao Tong University, Shanghai 200240, China}%
\email{jghuang@sjtu.edu.cn}%

\thanks{The first author was supported by the National Science Foundation (NSF) DMS-1418934 and in part by the Sea Poly Project of Beijing Overseas Talents.}
\thanks{The second author was partially supported by NSFC (Grant no. 11571237).}

\begin{abstract}
Some error analysis on virtual element methods (VEMs) including inverse inequalities, norm equivalence, and interpolation error estimates is presented for polygonal meshes each of which admits a virtual quasi-uniform triangulation. The related mesh regularity covers the usual one for theoretical analysis of VEMs and the proofs are only based on mathematical tools well-used in finite element methods technically.
\vskip 0.5cm

\vskip 0.3cm


\end{abstract}
\maketitle


\section{Introduction}
Since the pioneer work in \cite{Ahmad2013,BeiraoDaVeiga2012,BeiraoDaVeiga2014},
virtual element methods (VEMs) have widely been used for numerically solving various partial differential equations in recent years. Compared with the standard finite element methods (cf. \cite{Brenner.S;Scott.L2008,Ciarlet.P1978}), such methods have several significant advantages: (1) they are very adapted to polygonal/polyhedral meshes, leading to great convenience in mesh generation for problems with complex geometries. For example, in \cite{Chen2017} a simple and efficient interface-fitted polyhedral mesh algorithm is developed and VEM has been successfully applied to the elliptic interface problem. (2) They are very suitable for attacking the problems with high regularity solutions. For instance, it is very difficult to construct usual $H^2$-smooth finite element methods for fourth-order elliptic problems, and hence many nonconforming elements were devised to overcome the difficulty technically (cf. \cite{ShiWang2013}). It is, however, very convenient to construct $H^2$-smooth virtual element methods for this problem (cf. \cite{Brezzi2013}).  Until now, there have developed conforming and nonconforming VEMs for elliptic problems very sophistically (cf. \cite{Ahmad2013,BeiraoDaVeiga2012,BeiraoDaVeiga2016,Brezzi2013,CangianiManziniSutton2017,Chen2017,Dios2014}).


Error estimates for approximation spaces, inverse inequality, and norm equivalence between the norm of a finite element function and its degrees of freedom play fundamental roles in theoretical analysis of finite element methods. So are the virtual element methods.
Such results were stated or implied in the papers \cite{Ahmad2013,BeiraoDaVeiga2012}, though the detailed justifications were not presented.
More recently, in the papers \cite{BeiraoDaVeiga2015} and \cite{CangianiManziniSutton2017}, the inverse estimates (cf. (4.9) and (4.11) in \cite{BeiraoDaVeiga2015}) and  the norm equivalence (cf. Lemma 4.9 in \cite{CangianiManziniSutton2017}) were derived in detail, respectively. All these results were obtained using the so-called generalized scaling argument (cf. \cite{BoffiBrezziFortin2013}), based on the following assumptions on the polygon mesh $\mathcal{T}_h$ in two-dimensional cases:
\begin{enumerate}
\item[{\bf C1}.] There exists a real number $\gamma>0$ such that, for each element $K\in \mathcal {T}_h$, it is star-shaped with respect to a disk of radius $\rho_K\ge\gamma h_K$, where $h_K$ is the diameter of $K$.

\item[{\bf C2}.] There exists a real number $\gamma_1>0$ such that, for each element $K\in \mathcal{T}_h$, the distance between any two vertices of $K$ is $\ge \gamma_1 h_K$.
\end{enumerate}
Using the similar arguments in \cite{Dupont.T;Scott.R1980}, these estimates still hold if any element $K\in \mathcal{T}_h$ is the union of a finite number of polygons satisfying conditions {\bf C1} and {\bf C2}.

The key idea of the generalized scaling argument (still called the scaling argument in \cite{BoffiBrezziFortin2013}) is the use of the compactness argument. To fix ideas, let us show how to prove the inverse estimate
\begin{equation}
\label{inverse1}
\|v\|_{1,K}\le C h_K^{-1}\|v\|_{0,K} \quad \forall\,v\in V_K,
\end{equation}
where $V_K$ is a finite dimensional space of shape functions defined over a polygon $K\in \mathcal{T}_h$, and $C$ is a generic constant independent of the mesh size $h_K$. We first make a scaling transformation and rewrite \eqref{inverse1} as an equivalent estimate over the polygon $\hat{K}$ which is the image of $K$ after the previous transformation. In other words, it suffices to derive the estimate \eqref{inverse1} provided that $h_K=1$. In this case, under the assumptions of {\bf C1} and {\bf C2}, the set $\mathcal{K}$ consisting of all such $K$ can be viewed as a compact set in some topology.   Then, let
\begin{equation}
\label{inverse2}
f(K)=\sup_{v\in V_K} \frac {\|v\|_{1,K}} {\|v\|_{0,K}}.
\end{equation}
If we can prove that $f(K)$ is continuous with respect to $K\in \mathcal{K}$ in the sense of the above topology, then it is evident that $f(K)$ can attain its maximum $C$ over $\mathcal{K}$, leading to the desired estimate $\eqref{inverse1}$ readily; see Lemma \ref{lm:g} for such arguments.

Hence, if we apply the generalized scaling argument to derive the estimate \eqref{inverse1} for virtual element spaces, since $V_K$ is defined with the help of the Laplacian operator (for details see \cite{Ahmad2013,BeiraoDaVeiga2012} or Section 2), we require to show the solution of the Poisson equation defined over $K$ depends on the shape of $K$ continuously. In fact, such results may be obtained rigorously in a very subtle and technical way.

Based on the above comments, in this paper, we aim to derive all the results mentioned above only using mathematical tools well-used in the community of finite element methods, to shed light on theoretical analysis of virtual element methods in another way. We impose the following mesh regularity:
\begin{enumerate}
\item[{\bf A1}.] Every mesh $\mathcal T_h$ consists of a finite number of simple polygons (i.e. open simply connected sets with non-self-intersecting polygonal boundaries).

\item[{\bf A2}.] For each $K\in \mathcal T_h$, there exists a ``virtual triangulation" $\mathcal T_K$ of $K$ such that $\mathcal T_K$ is uniformly shape regular and quasi-uniform. The corresponding mesh size of $\mathcal T_K$ is proportional to $h_K$. Each edge of $K$ is a side of a certain triangle in $\mathcal T_K$.
\end{enumerate}
It is evident that the mesh $\mathcal{T}_h$ fulfilling the conditions {\bf C1} and {\bf C2} naturally satisfy the above conditions. We shall derive some error analysis on VEMs including inverse inequality, norm equivalence, and interpolation error estimates for several types of VEM spaces, under the mesh regularity conditions {\bf A1} and {\bf A2} which cover the usual ones frequently used in the analysis of virtual element methods.

For triangular meshes, one can use affine maps to map an arbitrary triangle to a so-called reference triangle and then work on the reference triangle. Results established on the reference triangle can be pulled back to the original triangle by estimating the Jacobian of the affine map. For polygons, scaling can be still used but not the affine maps. Therefore we cannot work on a reference polygon which does not exist for a family of general shape polygons. Instead we decompose a polygon $K$ into shape regular triangles and use the scaling argument in each triangle.


We will assume {\bf A1} and {\bf A2} hold throughout the paper. Most results are established on a generic polygon for which we always assume {\bf A2} holds. Constants hidden in the $\lesssim$ notation usually depends only on the shape regularity and quasi-uniformity of the auxiliary triangulation $\mathcal T_K$ assumed in {\bf A2}. Moreover, for any two quantities $a$ and $b$, ``$a\eqsim b$" indicates ``$a\lesssim b\lesssim a$". We will also use the standard notation and symbols for Sobolev spaces and their norms and semi-norms; we refer the reader to \cite{Adams1975} for more details.

Denote by $V_K$ a virtual element space (precise definition and variants of VEM spaces can be found in Section 2). With the help of {\bf A2}, we are going to rigorously prove that: for all $v\in V_K$
\begin{itemize}
\item Inverse inequality: $\|v\|_{1,K}\lesssim h_K^{-1}\|v\|_{0,K}$.
\smallskip

\item Norm equivalence:
$
h_K \|\boldsymbol  \chi (v)\|_{l^2} \lesssim \|v\|_{0,K} \lesssim h_K \|\boldsymbol  \chi (v)\|_{l^2},$
where $\boldsymbol  \chi(v)$ is the vector formed by the degrees of freedom of $v$.

\item Norm equivalence of VEM formulation:
\begin{align*}
\|\nabla v\|_{0,K}^2 &\eqsim \|\nabla \Pi_k^{\nabla} v\|_{0,K}^2 + \|\boldsymbol  \chi (v - \Pi_k^{\nabla} v)\|_{l^2}^2,\\
\|\nabla v\|_{0,K}^2 &\eqsim \|\nabla \Pi_k^{\nabla} v\|_{0,K}^2 + \|\boldsymbol  \chi_{\partial K} (v - \Pi_k^0 v)\|_{l^2}^2,
\end{align*}
where $\Pi_k^{\nabla},\Pi_k^0$ are $H^1, L^2$-projection to polynomial spaces, respectively.

\item Interpolation error estimate: if $I_k u\in V_K$ denotes the canonical interpolant defined by d.o.f. of $u$, then
$$
\|u - I_Ku\|_{0,K} + h_{K}|u - I_Ku|_{1,K} \lesssim h_K^{k+1}\|u\|_{k+1,K} \quad \forall u\in H^{k+1}(K).
$$
\end{itemize}

The rest of the paper is organized as follows. The virtual element method is introduced in Section 2 for later requirement. Inverse estimates, norm equivalence, and interpolation error estimates for several types of VEM spaces are derived technically in Sections 3-5, respectively.

\section{Virtual Element Methods}
We consider a two dimensional domain $\Omega$ which is decomposed into a polygon mesh $\mathcal T_h$ satisfying {\bf A1}. Namely each element in $\mathcal T_h$ is a simple polygon and a generic element will be denoted by $K$. We use two dimensional case for the clear illustration and will comment on the generalization to high dimensions afterwards.

To present the main idea, we consider the simplest Poisson equation with zero Dirichlet boundary condition:
$$
-\Delta u = f \text{ in }\Omega, \quad u|_{\partial \Omega} = 0.
$$
The weak formulation is: given an $f\in L^2(\Omega)$, find $u\in H_0^1(\Omega)$ such that
\begin{equation}\label{eq:weakform}
a(u,v) :=(\nabla u, \nabla v) = (f, v)\quad \forall v\in H_0^1(\Omega).
\end{equation}

\subsection{Assumptions on the polygon mesh}
%
%
%
As we mentioned in the introduction, we shall carry out our analysis based on the assumptions {\bf A1} and {\bf A2}. We give more discussion here.

Recall that a triangle is shape regular if there exists a constant $\kappa$ such that the ratio of the diameter of this triangle and the radius of its inscribed circle is bounded by $\kappa$. It is also equivalent to the condition that the minimum angle is bounded below by a positive constant $\theta$. A triangulation $\mathcal T$ is quasi-uniform if any two triangles in the triangulation is of comparable size. Namely there exists a constant $\sigma$, such that $\max_{\tau\in \mathcal T}h_{\tau} \leq \sigma \min_{\tau\in \mathcal T}h_{\tau}$. Uniformity means the constants $\kappa, \theta$ and $\sigma$ are independent of $K$.

By assumption {\bf A2}, the number of triangles of each `virtual triangulation' $\mathcal T_K$ is uniformly bounded by a number $L$ and the size of each triangle is comparable to that of the polygon, i.e. $h_{K}\lesssim h_{\tau}\leq h_K, \; \forall \tau \in \mathcal T_K$. The constants in our inequalities will depend on the shape regularity constant $\kappa$ (or equivalently $\theta$) and the quasi-uniformity constant $\sigma$ (or equivalently the number $L$).

Assumption {\bf A2} is introduced so that we can use estimates for finite elements on triangles. If we assume $K$ is star shaped and each edge is of comparable size, e.g. assumption {\bf C2}, a virtual triangulation can be obtained by connecting vertices of $K$ to the center of the star. But {\bf A2} allows the union of star shaped regions to form very irregular polygons.

Note that to have such a virtual triangulation, we can add more vertexes inside $K$ but not on the boundary $\partial K$.


\subsection{Spaces in Virtual Element Methods}
Let $k,l$ be two positive integers. We introduce the following space on $K$
\begin{equation}\label{VEMspace}
V_{k,l}(K) := \{ v\in H^1(K): v|_{\partial K} \in \mathbb  B_k(\partial K), \Delta v \in \mathbb  P_{l}(K)\},
\end{equation}
where $\mathbb  P_{k}(D)$ is the space of polynomials of degree $\leq k$ on $D$ and conventionally $\mathbb  P_{-1}(D) = 0$, and the space on the boundary
$$
\mathbb  B_k(\partial K):= \{ v\in C^{0}(\partial K): v|_e\in \mathbb  P_{k}(e) \text{ for all edges }e\subset \partial K\}.
$$
Namely restricted to $\partial K$, it is a standard conforming Lagrange element of degree $k$.
The shape function in \eqref{VEMspace} is well defined but the point-wise value of  a function $v\in V_k(K)$ requires solving a PDE inside $K$ and thus considered as implicitly defined not explicitly known. The novelty of VEM is that not the point-wise value but only the degree of freedom (d.o.f.) is enough to produce an accurate and stable numerical method.

To present the d.o.f., we first introduce a scaled monomial $\mathbb  M_{l}(D)$ on a $d$-dimensional domain $D$
\begin{equation}\label{eq:M}
\mathbb  M_{l} (D):= \left \{ \left ( \frac{\boldsymbol  x - \boldsymbol  x_c}{h_D}\right )^{\boldsymbol  s}, |\boldsymbol  s|\leq l\right \}
\end{equation}
with $h_D$ the diameter of $D$ and $\boldsymbol  x_c$ the centroid of $D$. When $D$ is a polygon, $\boldsymbol  x_c$ is the average of coordinates of all vertices of $D$ and thus $|\boldsymbol  x - \boldsymbol  x_c|\leq h_D$ for all $\boldsymbol  x\in D$.

We then introduce the dual space
\begin{equation}\label{VEM}
\mathcal X_{k,l}(K) = {\rm span}\{ \chi_a,  \chi_e^{k-2}, \chi_K^{l}\},
\end{equation}
where the functional vectors are
\begin{itemize}
\item $\chi_a$: the values at the vertices of $K$;
\smallskip
\item $\chi_e^{k-2}$: the moments on edges up to degree $k-2$
$$
\chi_e (v) = |e|^{-1}(m, v)_{e} \quad \forall m\in \mathbb  M_{k-2}(e), \forall \text{ edge } e\subset \partial K;
$$
\item $\chi_K^{l}$: the moments on element $K$ up to degree $l$
$$
\chi_{K}(v) = |K|^{-1}(m, v)_K \quad \forall m\in \mathbb  M_{l}(K).
$$
\end{itemize}

The verification
\begin{equation}\label{unisolvence}
(V_{k,l}(K)) ' = \mathcal X_{k,l}(K),
\end{equation}
is called unisovlence and has been established in~\cite{BeiraoDaVeiga2012}. For the completeness we give a different proof as follows.

It is easy to verify the dimensions matches, i.e., $\dim V_{k,l} = \dim \mathcal X_{k,l}$. Therefore it suffices to verify the uniqueness. That is, for $v\in V_{k,l}$, if $\chi(v) = 0$ for all $\chi \in \mathcal X_{k,l}$, then $v = 0$.

For $v\in V_{k,l}(K)\cap H_0^1(K)$, we apply the integration by parts to conclude
$$
(\nabla v, \nabla v)_K = (v, -\Delta v)_K = (Q_{l}v, -\Delta v)_K,
$$
where $Q_{l}$ is the $L^2$-orthogonal projection onto $\mathbb P_{l}(K)$. The last identity holds due to the requirement $\Delta v\in \mathbb P_{l}(K)$. Now the condition $\chi(v)=0$ for all $\chi \in \mathcal X_{k,l}$ implies that $v|_{\partial K} = 0$ and $Q_{l}v = 0$. Therefore $v\in H_0^1(K)$ and $\|\nabla v\|= 0$ which implies $v=0$.

\begin{remark}\rm
The operator $\Delta$ used in the definition of VEM space \eqref{VEMspace} can be replaced by other operators as long as the space $V_{k,l}(K)$ contains a polynomial space with appropriate degree, which ensures the approximation property. For example, when $K$ is triangulated into a triangulation $\mathcal T_K$, we can choose the standard $k$-th order Lagrange space on $\mathcal T_K$ and impose $\Delta_h v \in \mathbb  P_{l}(K)$ where $\Delta_h$ is the standard Galerkin discretization of $\Delta$ in the standard Lagrangian finite element space $S_{k}(\mathcal T_K)$ based on this virtual triangulation $\mathcal T_K$. From this point of view, VEM can be viewed as a kind of up-scaling. $\Box$
\end{remark}

We relabel the d.o.f. by a single index $i = 1, 2,  \ldots, N_{k,l} := \dim V_{k,l}(K)$. Associated with each d.o.f., there exists a basis $\{\phi_j\}$ of $V_{k,l}(K)$ such that $\chi_i(\phi_j) = \delta_{ij}$ for $i,j = 1, \ldots, N_{k,l}$. Then every function $v\in V_{k,l}(K)$ can be expanded as
$$
v(x) = \sum_{i=1}^{N_{k,l}}\chi_i(v)\phi_i(x)
$$
and in numerical computation it can be identified to the vector $\boldsymbol  v \in \mathbb  R^{N_{k,l}}$ in the form
$$
\boldsymbol  v = (\chi_1(v), \chi_2(v), \ldots, \chi_{N_{k,l}}(v))^{\intercal}.
$$
The isomorphism can be denoted by
$$
\boldsymbol  \chi: V_{k,l}(K) \to \mathbb  R^{N_{k,l}}, \quad \boldsymbol  \chi (v) =  (\chi_1(v), \chi_2(v), \ldots, \chi_{N_{k,l}}(v))^{\intercal}.
$$
The inverse of this isomorphism will be denoted by
$$
\Phi: \mathbb  R^{N_{k,l}} \to V_{k,l}(K), \quad \Phi (\boldsymbol  v) = \boldsymbol  \phi \cdot \boldsymbol  v,
$$
if we treat the basis as a vector $\boldsymbol  \phi = (\phi_1, \phi_2, \ldots, \phi_{N_{k,l}})^{\intercal}$.
%
%

Among different choices of the index $(k,l)$ in $V_{k,l}(K)$, the first VEM space in \cite{BeiraoDaVeiga2012} is
\begin{equation}\label{eq:Vk}
V_k(K) = V_{k,k-2}(K).
\end{equation}
Later on, in order to compute the $L^2$-projection of VEM functions, the authors of~\cite{Ahmad2013} introduced a larger space
\begin{equation}\label{eq:Vkk}
\widetilde{V}_k (K) = V_{k,k}(K)
\end{equation}
and a subspace isomorphism to $V_k(K)$
\begin{equation}\label{eq:Wk}
W_k(K) = \{ w\in \widetilde{V}_k (K) : (w - \Pi_k^{\nabla}w, q^*)_K = 0 \quad \forall q^*\in \mathbb  M_{k}(K)\backslash  \mathbb  M_{k-2}(K)\},
\end{equation}
where the $H^1$-projection $\Pi_k^{\nabla}$ will be defined in the next section.

The spaces $V_k(K)$ and $W_k(K)$ are different but share the same d.o.f.
For the same vector $\boldsymbol  v\in \mathbb  R^{N_{k,k-2}}$, we can then have different functions $\Phi_{V}(\boldsymbol  v)\in V_k(K)$ and $\Phi_{W}(\boldsymbol  v) \in W_k(K)$ and in general $\Phi_{V}(\boldsymbol  v) \neq \Phi_{W}(\boldsymbol  v)$.

Function spaces in each element will be used to design a virtual element space on the whole domain $\Omega$ in the standard way since the function is continuous across the boundary of elements. In particular, given a polygon mesh $\mathcal T_h$ of $\Omega$ and a given integer $k\geq 1$, we define
\begin{align*}
V_h^{k,l}  &= \{ v\in H^1(\Omega): v|_{K} \in V_{k,l}(K)\quad \forall K\subset \mathcal T_h\},\\
V_h  &= \{ v\in H^1(\Omega): v|_{K} \in V_k(K)\quad \forall K\subset \mathcal T_h\},\\
\widetilde{V}_h  &= \{ v\in H^1(\Omega): v|_{K} \in \widetilde{V}_k (K)\quad \forall K\subset \mathcal T_h\},\\
W_h  &= \{ v\in H^1(\Omega): v|_{K} \in W_k (K)\quad \forall K\subset \mathcal T_h\}.
\end{align*}
The d.o.f. can be defined for the global space in the natural way.

For pure diffusion problem, $V_h$ is enough. The function spaces $W_h$ and $\widetilde{V}_h$  will be helpful to deal with low order terms in, e.g., reaction-diffusion problems, and simplify the implementation in three dimensions (cf. \cite{Ahmad2013}).

\subsection{Approximate Stiffness Matrix}
A conforming virtual finite element space $V_h^0 : = V_h \cap H_0^1(\Omega)$ is chosen to discretize \eqref{eq:weakform}. We cannot, however, compute the Galerkin projection of $u$ to $V_h^0$ since the traditional way of computing $a(u_h, v_h)$ using numerical quadrature requires point-wise information of functions and their gradient inside each element. In virtual element methods, only d.o.f is used to assemble an approximate stiffness matrix.

\smallskip

Define a local $H^1$ projection  $\Pi_k^{\nabla}: H^1(K)\to \mathbb  P_k(K)$ as follows: given $v \in H^1(K)$, let $\Pi_k^{\nabla} v \in \mathbb  P_k(K)$ satisfy
$$
(\nabla \Pi_k^{\nabla} v, \nabla p)_K = (\nabla v, \nabla p)_K, \quad \text{for all } p \in \mathbb  P_k(K).
$$
The right hand side can be written as
$$
(\nabla v, \nabla p)_K =  - (v, \Delta p)_K + \langle v, n\cdot \nabla p \rangle_{\partial K}.
$$
When $v$ is in a VEM space with $l\geq k-2$ (either $V_k(K), \widetilde{V}_k(K)$ or $W_k(K)$), it can be computed using d.o.f. of $v$ since, for $p \in \mathbb  P_k(K)$, $\Delta p\in \mathbb  P_{k-2}(K)$ and $\nabla p\cdot n\in \mathbb  P_{k-1}(e), e\in \partial K$. The operator $\Pi_k^{\nabla}$ can be naturally extended to the global space $V_h^{k,l}$ piece-wisely.

\begin{remark}\rm
As $\nabla p\cdot n\in \mathbb  P_{k-1}(e), e\in \partial K$, if we do not enforce the continuity of $v|_{\partial K}$, we could discard the d.o.f. on vertices and use d.o.f. of the edge moments up to order $k-1$ which leads to a non-conforming VEM (cf. \cite{Dios2014}) or weak Galerkin methods (cf. \cite{MuWangYe2015polyredu}).
\end{remark}


As $(\nabla \cdot, \nabla \cdot)$ is only semi-positive definite, a constraint should be imposed to eliminate the constant kernel. When $\Pi^{\nabla}_k$ is applied to a VEM function, we shall choose the constraint
$$
\int_{K} v \dx = \int_K \Pi^{\nabla}_k v \dx, \quad \text{ if } l \geq 0
$$
or in the lowest order case $l = -1$
$$
\int_{\partial K} v \dd s = \int_{\partial K} \Pi^{\nabla}_k v \dd s.
$$
Both constraints can be imposed using the d.o.f. of a VEM function.

For later uses, let us next recall the following Poincar\'e-Friedrichs inequality for $u\in H_0^1(K)$
$$
 \|u\|_{0,K}\leq h_{K}\| \nabla u \|_{0,K}.
$$
and the following version established in~\cite{Brenner.S2003a}.

\begin{lemma}[Poincar\'e-Friedrichs inequality~\cite{Brenner.S2003a}] \label{lm:PF}
We have the following Poincar\'e-Friedrichs inequality
\begin{equation}\label{eq:Poincare}
\| u - \Pi_k^{\nabla} u\|_{0,K}\lesssim h_K\| \nabla (u - \Pi_k^{\nabla} u)\|_{0,K} \quad \forall u\in H^1(K)
\end{equation}
where the constant depends only on the shape regularity constant of the triangulation $\mathcal T_K$.
\end{lemma}
The scaling factor $h_K$ is not presented in the form in~\cite{Brenner.S2003a} but can be easily
obtained by the following scaling argument.
We apply the transformation $\hat {\boldsymbol  x} = (\boldsymbol  x- \boldsymbol  x_c)/h_{K}$ so that $\hat K$,
the image of $K$, is contained in the unit disk. The transformed triangulation $\hat T_{\hat K}$ is
still shape regular so that we can apply results in~\cite{Brenner.S2003a}.
Then scale back to $K$ to get the constant $h_K$.
As pointed out in~\cite{Brenner.S2003a}, the constant depends only on the shape regularity not the quasi-uniformity of the triangulation $\mathcal T_K$.

The first part of the approximate stiffness matrix of the virtual element method will be obtained by the following bilinear form
$$
a( \Pi_k^{\nabla} u, \Pi_k^{\nabla} v).
$$


\subsection{Stabilization}
The approximate bilinear form $a( \Pi_k^{\nabla} u, \Pi_k^{\nabla} v)$ alone will not lead to a stable method. Since $\mathbb  P_k(K) \subset V_k(K)$ and it is a strict subspace except the case $K$ is a triangle, we may have $a( \Pi_k^{\nabla} v, \Pi_k^{\nabla} v) = 0$ when $v\in \ker( \Pi_k^{\nabla})/\mathbb  R$. Namely $a( \Pi_k^{\nabla} \cdot, \Pi_k^{\nabla} \cdot)$ alone cannot define an inner product on $V_h^0$.

A stabilization term should be added to gain the coercivity. To impose the stability while maintain the accuracy, the following assumptions on the element-wise stabilization term $S_{K}(\cdot,\cdot)$ are imposed in VEM (cf. \cite{BeiraoDaVeiga2012}).
\begin{itemize}
\item $k$-consistency: for $p_k\in \mathbb  P_k(K)$
$$
S_{K}(p_k, v) = 0 \quad \forall v\in V_h.
$$
\item stability:
$$
S_{K}(\tilde u, \tilde u) \eqsim (\nabla \tilde u, \nabla \tilde u)_K \quad \forall \tilde u\in (I - \Pi_k^{\nabla})V_h.
$$
\end{itemize}
So VEM is indeed a family of schemes different in the choice of stabilization terms.

We then define
$$
a_h(u,v) := a(\Pi_k^{\nabla} u, \Pi_k^{\nabla} v) + \sum_{K\in \mathcal T_h} S_{K}(u,v).
$$

The $k$-consistency will imply the {\em Patch Test}, i.e., if $u\in \mathbb  P_k(\Omega)$, then
$$
a(u, v_h) = a_h(u, v_h), \quad \text{for all } v_h\in V_h.
$$
The stability will imply
$$
a(u,u) \eqsim a_h(u, u) \quad \text{for all } u\in V_h.
$$
An abstract error estimate of VEM with stabilization satisfying $k$-consistency and stability is given in~\cite{BeiraoDaVeiga2012}.

In the continuous level, a stabilization term can be a scaled $L^2$-inner product
\begin{equation}\label{eq:scaledL2}
h_K^{-2}( u - \Pi_k^{\nabla}u, v - \Pi_k^{\nabla}v )_K.
\end{equation}
The $k$-consistency is obvious as $\Pi_k^{\nabla}$ preserves polynomials of degree $\leq k$. The stability can be proved using an inverse inequality and Poincar\'{e}-Friedrichs type inequality and will be proved rigorously later on.

In the implementation, stabilization \eqref{eq:scaledL2} is realized as
\begin{equation}\label{eq:SVEM}
S_{\boldsymbol  \chi}(u,v) := \boldsymbol  \chi ((I - \Pi_k^{\nabla})u) \cdot \boldsymbol  \chi ((I - \Pi_k^{\nabla})v).
\end{equation}
That is we use the $l^2$-inner product of the d.o.f. vectors to approximate the $L^2$-inner product of the functions involved. The scaling factor $h_K^{-2}$ is absorbed into the definition of d.o.f. through the scaling of the monomials ( cf. \eqref{eq:M}). The norm equivalence of $l^2$ and $L^2$ norm is well known for standard finite element spaces. Rigorous justification for functions in VEM spaces will be established in Section \S \ref{sec:norm} (see also Lemma 4.9 in \cite{CangianiManziniSutton2017}).

\section{Inverse Inequalities}
In this section we shall establish the inverse inequality
$$
\|\nabla v\|_{0,K} \leq C h_K^{-1}\|v\|_{0,K}\quad \text{for all } v\in V_{k,l}(K).
$$
As we mentioned in the introduction, one approach is to apply the fact that all norms are equivalent on the finite dimensional space $V_{k,l}(K)$. How the constant $C$ depends on the shape of $K$ is, however, not clear by a simple scaling argument.

To overcome the above difficulty, we shall use a shape regular and quasi-uniform `virtual triangulation' $\mathcal T_K$ and the fact $\Delta v\in \mathbb P_{l}$. Note that if we modify the definition of virtual element spaces by using the discrete Laplacian operator, then the inverse inequality is trivially true as now the function in VEM space is a finite element function on the virtual triangulation.

We first establish an inverse inequality for polynomial spaces on polygons.
\begin{lemma}[Inverse inequality of polynomial spaces on a polygon]\label{lm:invpoly}
There exists a generic constant depends only on the shape regularity and quasi-uniformity of $\mathcal T_K$ s.t.
 $$
 \| g \|_{0,K} \lesssim h_K^{-i}\| g \|_{-i,K} \quad \text{ for all } g\in \mathbb P_k, i = 1,2.
 $$
\end{lemma}
\begin{proof}
Restricted to one triangle $\tau \in \mathcal T_K$, noting $g$ is a polynomial and using the scaling argument, we have $ \|g\|_{0,\tau} \lesssim h_{\tau}^{-i}\|g\|_{-i,\tau}$, for $i=1,2$, where the constant depends only on the shape regularity.
By definition of the dual norm, $\|g\|_{-i,\tau} \leq \|g\|_{-i,K}$. Therefore
$$
 \|g\|_{0,K}^2 =  \sum_{\tau\in \mathcal T_K} \|g\|_{0,\tau}^2 \lesssim  \sum_{\tau\in \mathcal T_K} h_{\tau}^{-2i}\|g\|_{-i, \tau}^2\lesssim  h_K^{-2i}\|g\|_{-i,K}^2,
$$
as required.
\end{proof}

Let $S_{k}(\mathcal T_K)$ be the standard continuous $k$-th Lagrange finite element space on $\mathcal T_K$ and $S_{k}^{0}(\mathcal T_K) := S_{k}(\mathcal T_K) \cap H_0^1(K)$.
Define $Q_K: V_{k,l}(K) \to S_{k}(\mathcal T_K)$ as follows:
\begin{enumerate}
\item $Q_K v|_{\partial K} = v|_{\partial K}$;
\smallskip
\item $(Q_K v, \phi)_K = (v, \phi)_K$ for all $\phi \in S_{k}^{0}(\mathcal T_K)$.
\end{enumerate}
Namely we keep the boundary value and obtain the interior value by the $L^2$-projection. We need the following stability result of $Q_K$.

\begin{lemma}[Weighted stability of $Q_K$]\label{lm:QK}
For any $\epsilon >0$
$$
h_K^{1/2}\|Q_K v\|_{0,\partial K} + \|Q_K v\|_{0,K}\lesssim (1+ \epsilon^{-1}) \|v\|_{0,K} + \epsilon h_K \|\nabla v\|_{0,K}, \quad v\in V_{k,l}(K).
$$
\end{lemma}
\begin{proof}
We split $Q_K v = v_{\partial,h} + v_{0,h}$, where $v_{\partial, h}\in S_{k}(\mathcal T_K)$ is uniquely determined by  $v_{\partial, h}|_{\partial K} = Q_K v |_{\partial K} = v |_{\partial K}$ and vanishes on other nodes of $S_{k}(\mathcal T_K)$. Consequently $v_{0,h} = Q_K v - v_{\partial,h}\in S_{k}^{0}(\mathcal T_K)$. Then
$$
(Q_K v, Q_K v)_K = (Q_K v, v_{\partial,h})_K  + (Q_K v, v_{0,h})_K =: {\rm I}_1 + {\rm I}_2.
$$
The first term can be bound by
$$
{\rm I}_1\leq \|Q_Kv\|_{0,K}\|v_{\partial,h}\|_{0,K}.
$$
By the definition of $Q_K$, we can bound the second term as
$$
{\rm I}_2 = (v, v_{0,h})_K \leq \|v\|_{0,K}\|v_{0,h}\|_{0,K} \leq \|v\|_{0,K}\left (\|v_{\partial,h}\|_{0,K} + \|Q_K v\|\right ).
$$
By Young's inequality, we can then obtain the inequality
\begin{equation}\label{eq:Qhv}
\|Q_K v\|_{0,K}\lesssim \|v\|_{0,K} + \|v_{\partial,h}\|_{0,K}.
\end{equation}

So the key is to estimate the boundary term $\|v_{\partial,h}\|_{0,K}$. For a boundary edge $e$, denote by $\tau_e$ the triangle in $\mathcal T_K$ with $e$ as an edge. By the definition of $v_{\partial,h}$, we have
$$
\|v_{\partial,h}\|_{0,K}^2 = \sum_{e\subset \partial K}\|v_{\partial,h}\|_{0,\tau_e}^2 \lesssim \sum_{e\subset \partial K}\|v_{\partial,h}\|_{0,e}^2h_e = \sum_{e\subset \partial K}\|v\|_{0,e}^2h_e.
$$
In the last step, we use the fact $v_{\partial, h}|_{\partial K} = Q_K v |_{\partial K} = v |_{\partial K}$.

On the other hand, for a bounded domain $\omega$ with Lipschitz boundary, we have the estimate $\|v\|_{0,\partial \omega}^2\lesssim \|v\|_{0,\omega} \|\nabla v\|_{0,\omega}$ for any $v\in H^1(\omega)$ (cf. \cite{Brenner.S;Scott.L2008}). Hence, it follows from the scaling argument and Young's inequality that on each triangle $\tau_e$, there holds the following weighted trace estimate
$$
\|v\|_{0,e}^2h_e \lesssim \epsilon^{-2} \|v\|_{0,\tau_e}^2 + \epsilon^2 h_e^2 \|\nabla v\|_{0,\tau_e}^2.
$$
Summing over $e\subset \partial K$ and taking square root, we obtain
\begin{equation}
\label{add1}
\|v_{\partial,h}\|_{0,K} \lesssim\|Q_Kv\|_{0,\partial K}h_K^{1/2} \lesssim \epsilon^{-1} \|v\|_{0,K} + \epsilon h_K \|\nabla v\|_{0,K},
\end{equation}
and substitute it into \eqref{eq:Qhv} to get the desired inequality for $\|Q_K v\|_{0,K}$.
The desired estimate for $h_K^{1/2}\|Q_K v\|_{0,\partial K}$ follows from \eqref{add1} directly.
\end{proof}

To develop various estimates for a function in VEM spaces, we require to separate it into two functions, related to the moment and the trace of the function, respectively.
\begin{lemma}[An $H^1$-orthogonal decomposition]\label{lm:dec}
For any function $v\in H^1(K)$, we can decompose it as
$$
v =  v_1 + v_2,
$$
with
\begin{enumerate}
\item $v_1\in H^1(K), v_1 |_{\partial K} = v|_{\partial K}, \Delta v_1 = 0$ in $K$,
\smallskip
\item  $v_2 \in H_0^1(K), \Delta v_2 = \Delta v$ in $K$.
\end{enumerate}
Furthermore the decomposition is $H^1$-orthogonal in the sense that
$$
\|\nabla v\|^2 _{0,K}= \|\nabla v_1\|^2_{0,K} + \|\nabla v_2\|^2_{0,K}.
$$
\end{lemma}
\begin{proof}
 We simply choose $v_2$ as the $H^1$-projection of $v$ to $H_0^1(K)$, i.e., $v_2 \in H_0^1(K)$ and
 $$
 (\nabla v_2, \nabla \phi)_K  =  (\nabla v, \nabla \phi)_K \quad \text{for all }\phi \in H_0^1(K),
 $$
 and set $v_1 = v - v_2$. Equivalently we can take the trace of $v$ and apply harmonic extension to get $v_1$ and set $v_2 = v - v_1$.
\end{proof}

For the harmonic part, we have the following inequality.
\begin{lemma}[A weighted inequality of the harmonic part of a VEM function]\label{lm:v1}
For $v\in V_{k,l}(K)$, let $v_1\in H^1(K), v_1 |_{\partial K} = v|_{\partial K}, \Delta v_1 = 0$ in $K$. Then for any $\epsilon>0$ we have the following inequality of $v_1$
$$
\|\nabla v_1\|_{0,K} \lesssim h_K^{-1}(1+ \epsilon^{-1}) \|v\|_{0,K} + \epsilon \|\nabla v\|_{0,K}.
$$
\end{lemma}
\begin{proof}
Using the fact $\Delta v_1 =0$ in $K$, we have the property
\begin{equation}\label{eq:min}
\|\nabla v_1\|_{0,K} = \inf_{w\in H^1(K), w|_{\partial K} = v_1|_{\partial K}}\|\nabla w\|_{0,K}.
\end{equation}
Observe that $Q_K v|_{\partial K} = v_1|_{\partial K}$ and $Q_K v\in H^1(K)$. Therefore, from the principle of energy minimization (cf. \eqref{eq:min}), the inverse inequality for functions in $S_{k}(\mathcal T_K)$, and the weighted stability of $Q_K$, it follows that
\begin{align*}
\|\nabla v_1\|_{0,K} \leq \|Q_Kv\|_{0,K} \lesssim h_K^{-1}\|Q_Kv\|_{0,K}\leq h_K^{-1}(1+ \epsilon^{-1}) \|v\|_{0,K} + \epsilon \|\nabla v\|_{0,K}.
\end{align*}
 The proof is completed.
\end{proof}

We now estimate the second part in the decomposition.
\begin{lemma}[Inverse inequality of non-zero moments part]\label{lm:v2}
 For $v\in V_{k,l}(K)$, let $v_2 \in H_0^1(K)$ satisfies $\Delta v_2 = \Delta v$ in $K$. Then
 $$
\|\nabla v_2\|_{0,K} \lesssim h_K^{-1}\|v\|_{0,K}.
$$
\end{lemma}
\begin{proof}
As $v_2\in H_0^1(K)$, we can apply the integration by parts to get
$$
\|\nabla v_2\|_{0,K}^2 = -(\Delta v_2, v_2)_K = -(\Delta v, v_2)_K \leq \|\Delta v\|_{0,K}\|v_2\|_{0,K}.
$$
For $v\in V_{k,l}(K)$, we can apply the inverse inequality for $\Delta v\in \mathbb  P_l$:
$$
\|\Delta v\|_K \lesssim h_K^{-2}\|\Delta v\|_{-2,K} \leq h_K^{-2}\|v\|_{0,K}.
$$
Combining with the Poincar\'e-Friedrichs inequality for $v_2\in H_0^1$: $\|v_2\|_{0,K}\lesssim h_K\|\nabla v_2\|_{0,K}$, we then get
$$
\|\nabla v_2\|_{0,K}^2 \lesssim h_K^{-1}\|v\|\|\nabla v_2\|_{0,K},
$$
and cancel one $\|\nabla v_2\|_{0,K}$ to get the desired result.
\end{proof}

\begin{theorem}[Inverse inequality of a VEM function]\label{th:inverse}There exists a constant $C$ depending only on the shape regularity and quasi-uniformity of $\mathcal T_K$ such that
 $$
\|\nabla v\|_{0,K} \leq C h_K^{-1}\|v\|_{0,K}\quad \text{for all } v\in V_{k,l}(K).
$$
\end{theorem}
\begin{proof}
By Lemmas \ref{lm:v1} and \ref{lm:v2}, we have
$$
\|\nabla v\|_{0,K} \leq \|\nabla v_1\|_{0,K} + \|\nabla v_2\|_{0,K} \lesssim h_K^{-1}\|v\|_{0,K} + \epsilon \|\nabla v\|_{0,K}.
$$
Choose $\epsilon$ small enough and absorb the term $\epsilon \|\nabla v\|_{0,K}$ to the left hand side to get the desired inverse inequality.
\end{proof}

As an application of the inverse inequality, we prove the $L^2$-stability of the projection $Q_K$ and $\Pi_k^{\nabla}$ restricted to VEM spaces.

\begin{corollary}[$L^2$-stability of $Q_K$]\label{cor:QK}
The operator $Q_K: V_{k,l}(K) \to S_{k}(\mathcal T_K)$ is $L^2$-stable, i.e.,
$$
\|Q_Kv\|_{0,K} \lesssim \|v\|_{0,K}, \quad \text{for all } v\in V_{k,l}(K).
$$
\end{corollary}
\begin{proof}
Simply apply the inverse inequality to bound $h_K\|\nabla v\|_{0,K}\lesssim \|v\|_{0,K}$ in Lemma \ref{lm:QK} to get the desired result.
\end{proof}

\begin{corollary}[$L^2$-stability of $\Pi_k^{\nabla}$]\label{cor:PiK}
Let $k,l$ be two positive integers and $l\geq k-2$. The operator $\Pi_k^{\nabla}: V_{k,l}(K) \to \mathbb  P_{k}(K)$ is $L^2$-stable, i.e.,
$$
\|\Pi_k^{\nabla}v\|_{0,K} \lesssim \|v\|_{0,K}, \quad \text{for all } v\in V_{k,l}(K).
$$
\end{corollary}
\begin{proof}
 By the triangle inequality and the Poincar\'e-Friedrichs inequality, we have
 $$
\|\Pi_k^{\nabla}v\|_{0,K}\leq \|v\|_{0,K} + \| v- \Pi_k^{\nabla}v\|_{0,K} \lesssim \|v\|_{0,K} + h_K\| \nabla(v- \Pi_k^{\nabla}v)\|_{0,K}.
 $$
 Then by the $H^1$-stability of $\Pi_k^{\nabla}$ and the inverse inequality
 $$
 h_K\| \nabla(v- \Pi_k^{\nabla}v)\|_{0,K}\lesssim h_K\| \nabla v\|_{0,K}\lesssim \|v\|_{0,K}.
 $$
  The proof is thus completed.
\end{proof}

\section{Norm Equivalence}\label{sec:norm}
We shall prove a norm equivalence between $L^2$-norm of a VEM function and $l^2$-norm of the corresponding vector representation using d.o.f.
Consequently we obtain the stability of two stabilization choices used in VEM formulation.

\subsection{Norm equivalence of polynomial spaces on a polygon}
We begin with a norm equivalence of polynomial spaces on polygons.
\begin{lemma}[Norm equivalence of polynomial spaces on a polygon]\label{lm:g}
Let $g = \sum_{\alpha} g_{\alpha}m_{\alpha}$ be a polynomial on $K$. Denote by $\boldsymbol  g = (g_{\alpha})$ the coefficient vector. Then we have the norm equivalence
 $$
h_{K}\|\boldsymbol  g\|_{l^2} \lesssim \|g\|_{0,K}\lesssim h_K\|\boldsymbol  g\|_{l^2}.
 $$
\end{lemma}
\begin{proof}
 The inequality $\|g\|_{0,K}\lesssim h_K\|\boldsymbol  g\|_{l^2}$ is straightforward. As $\boldsymbol  x_c$ is the average of coordinates of all vertices of the polygon, we have $\|m_{\alpha}\|_{\infty, K}\leq 1$ and thus $\|m_{\alpha}\|_{0, K}\lesssim h_K$. Then by Minkowski's inequality and the Cauchy inequality,
 $$
 \|g\|_{0,K} \leq \sum_{\alpha} |g_{\alpha}|\|m_{\alpha}\|_{0,K} \lesssim h_K\|\boldsymbol  g\|_{l^2}.
 $$

The lower bound $h_{K}\|\boldsymbol  g\|_{l^2} \lesssim \|g\|_{0,K}$ is technical. Again we cannot apply the standard scaling argument since there is no reference polygon. Instead we chose a circle $S_{\tau}$ inside a triangle $\tau$ in the virtual triangulation $\mathcal T_{K}$ such that the radius satisfies $r_{\tau} = \delta h_K$, where the constant $\delta \in (0,1)$ depending only on the shape regularity and quasi-uniformity of the triangulation $\mathcal T_{K}$. We then apply an affine map $\hat {\boldsymbol  x} = (\boldsymbol  x- \boldsymbol  x_c)/h_{K}$. The transformed circle $\hat S_{\tau}$ with radius $\delta$ is contained in the unit disk with center origin. As $S_{\tau}\subset K$, we have
\begin{equation}\label{eq:scale}
\|g\|_{0,K} \geq \|g\|_{0,S_{\tau}} = \|\hat g\|_{0,\hat S_{\tau}} h_K,
\end{equation}
where $\hat g(\hat{\boldsymbol  x}) := g(\boldsymbol  x)$.
Let $\hat M_{ij} = \int_{\hat S_{\tau}}\hat m_i \hat m_j \dd \hat x$ and $\hat M = (\hat M_{ij})$. Then
\begin{equation}\label{eq:ghat}
\|\hat g\|_{0,\hat S_{\tau}}^2 = \boldsymbol  g^{\intercal} \hat M \boldsymbol  g \geq \lambda_{\min}(\hat M) \|\boldsymbol  g\|_{l^2}^2.
\end{equation}
The entry $\hat M_{ij}$ of the mass matrix is a continuous function of the center $\boldsymbol  c$ of the circle $\hat S_{\tau}$. So is $\lambda_{\min}(\hat M) = \lambda_{\min} (\boldsymbol  c)$. By our construction, $\boldsymbol  c$ is contained in the unit disk. We then let $\lambda^* = \min_{\boldsymbol  c, |\boldsymbol  c|\leq 1}\lambda_{\min} (\boldsymbol  c)$ and obtain a uniform bound
$\|\hat g\|_{0,\hat S_{\tau}}^2 \geq \lambda^*  \|\boldsymbol  g\|_{l^2}^2$. Notice that after the scaling, we work on a reference circle and thus the constant $\lambda^*$ will depend only on the radius $\delta$ of $\hat S_{\tau}$.

Combining \eqref{eq:scale} and \eqref{eq:ghat}, we obtain the desired inequality
$$
h_{K}\|\boldsymbol  g\|_{l^2} \lesssim \|g\|_{0,K},
$$
with a constant depending only the shape regularity and quasi-uniform constants of triangulation $\mathcal T_{K}$.
\end{proof}

\subsection{Norm equivalence for VEM spaces}
We shall prove the normal equivalence of the $L^2$-norm of VEM functions and the $l^2$-norm of their corresponding vectors.
\begin{lemma}[Lower bound]
 For $v\in V_{k,l}(K)$, we have the inequality
 $$
 h_K\|\boldsymbol  \chi (v)\|_{l^2} \lesssim \|v\|_{0,K}.
 $$
\end{lemma}
\begin{proof}
 We group the d.o.f. into two groups: $\boldsymbol  \chi_{\partial K}(\cdot)$ are d.o.f associated with the boundary of $K$ and $\boldsymbol  \chi_{K}(\cdot)$ are moments in $K$.

Restricted to the boundary $v|_{\partial K}\in \mathbb  B_k(K)$ which consists of standard Lagrange elements. So by the standard scaling argument, we have
$$
h_K\|\boldsymbol  \chi_{\partial K} (v)\|_{l^2} \eqsim h_K^{1/2}\|v\|_{0,\partial K}.
$$
Apply the weighted trace theorem (cf. Lemma \ref{lm:QK}), and the inverse inequality to functions in VEM spaces to obtain
$$
h_K^{1/2}\|v\|_{0,\partial K} \lesssim \|v\|_{0,K} + h_K\|\nabla v\|_{0,K}  \lesssim \|v\|_{0,K}.
$$

For the d.o.f. of interior moments, we apply the Cauchy-Schwarz inequality to obtain
$$
|K|^{-1}\int_K v m \dx \leq |K|^{-1}\|v\|_{0,K}\|m\|_{0,K} \lesssim h_K^{-1}\|v\|_{0,K}, \quad \text{for all }m \in \mathbb  M_l.
$$

Combining the estimate of $\boldsymbol  \chi_{\partial K}(\cdot)$ and $\boldsymbol  \chi_{K}(\cdot)$, we finish the proof.
\end{proof}

Estimate of the upper bound turns out to be technical. Again we shall use the $H^1$ decomposition presented in Lemma \ref{lm:dec}.
%
\begin{lemma}[Upper bound for the harmonic part]\label{lm:harmupp}
For any $v\in V_{k,l}(K)$, let $v_1\in H^1(K)$ satisfy $v_1|_{\partial K} = v|_{\partial K}$ and $\Delta v_1 = 0$ in $K$. Then
$$
\|v_1\|_{0,K} \lesssim h_K \|\boldsymbol  \chi_{\partial K} (v)\|_{l^2}.
$$
\end{lemma}
\begin{proof}
 We can write $$v_1 = \sum_{i=1}^{N_{\partial K}}\chi_{i} (v_1)\phi_i(x),$$
 where $\{\phi_i|_{\partial K}\}\subset \mathbb  B_{k}(\partial K)$ is a dual basis of $\boldsymbol  \chi_{\partial K}$ on the boundary and $\Delta \phi_i = 0$ inside $K$. By Minkowski's inequality, it suffices to prove $\|\phi_i\|_{0,K}\lesssim h_K$.

Restricted to the boundary, we can apply the scaling argument for each edge and conclude $\|\phi_i\|_{\infty,\partial K}\leq C$. As $\phi_i$ is harmonic, by the maximum principle, $\|\phi_i\|_{\infty,K} \leq \|\phi_i\|_{\infty,\partial K}\leq C$. Then $\|\phi_i\|_{0,K}\lesssim h_K$ follows.
\end{proof}

\begin{lemma}[Upper bound for the moment part]
For any $v\in V_{k,l}(K)$, let $v_2 \in H_0^1(K)$ satisfy $\Delta v_2 = \Delta v$ in $K$. Then
$$
\|v_2\|_{0,K} \lesssim h_K \|\boldsymbol  \chi (v)\|_{l^2}.
$$
\end{lemma}
\begin{proof}
Let $g = -\Delta v = -\Delta v_2$. Then by integration by parts
\begin{equation}\label{eq:v2}
\|\nabla v_2\|^2_{0,K} = -(\Delta v_2, v_2)_K = (g, v_2)_K = (g, v)_K - (g, v_1)_K.
\end{equation}
We expand $g$ in the basis $m_{\alpha}$ i.e. $g = \sum_{\alpha} g_{\alpha}m_{\alpha}$ and denote by $\boldsymbol  g = (g_{\alpha})$. Then by the Cauchy inequality and the normal equivalence for $g$ (cf. Lemma \ref{lm:invpoly}), we have
$$
(g, v)_K =  \sum_{\alpha} g_{\alpha} \chi_{\alpha}(v) \leq \|\boldsymbol  g\|_{l^2}\|\boldsymbol  \chi_{K}(v)\|_{l^2} \lesssim h_K^{-1}\|g\|_{0,K}\|\boldsymbol  \chi_{K}(v)\|_{l^2}.
$$
Substituting into \eqref{eq:v2}, we then obtain an upper bound of $\|\nabla v_2\|$ as
$$
\|\nabla v_2\|^2_{0,K} \lesssim h_K^{-1}\|\Delta v_2\|_{0,K} \left (\|\boldsymbol  \chi_{K}(v)\|_{l^2} +  \|v_1\|_{0,K} \right ) \lesssim \|\nabla v_2\|_{0,K} \|\boldsymbol  \chi(v)\|_{l^2}.
$$
In the last step, we have used the inverse inequality and the upper bound for $v_1$ established in Lemma \ref{lm:harmupp}. Canceling one $\|\nabla v_2\|_{0,K}$, we obtain the inequality
$$
\|\nabla v_2\|_{0,K} \lesssim \|\boldsymbol  \chi(v)\|_{l^2}.
$$

Finally we finish the proof by using the Poincar\'e inequality $\|v_2\|_{0,K}\lesssim h_K\|\nabla v_2\|_{0,K}$ for $v_2\in H_0^1(K)$.
\end{proof}

In summary, the following theorem holds.
\begin{theorem}[Norm equivalence between $L^2$ and $l^2$-norms]\label{th:norm}
For any $v\in V_{k,l}(K)$, we have the norm equivalence
$$
h_K \|\boldsymbol  \chi (v)\|_{l^2} \lesssim \|v\|_{0,K} \lesssim h_K \|\boldsymbol  \chi (v)\|_{l^2}.
$$
\end{theorem}

For functions in space $V_k(K)$, Theorem \ref{th:norm} can be applied directly. For space $W_k(K)\subset V_{k,k}(K)$, if we apply Theorem \ref{th:norm} for functions in $V_{k,k}(K)$, additional moments in $\boldsymbol  \chi_{K}^{k}\backslash \boldsymbol  \chi_{K}^{k-2}$ should be involved. We shall show that for $W_k(K)$, no additional moment is needed.
\begin{corollary}[Norm equivalence between $L^2$ and $l^2$-norms for $W_k(K)$]\label{cor:norm}
For any $v\in W_{k}(K)$, we have the norm equivalence
$$
h_K \|\boldsymbol  \chi (v)\|_{l^2} \lesssim \|v\|_{0,K} \lesssim h_K \|\boldsymbol  \chi (v)\|_{l^2}.
$$
\end{corollary}
\begin{proof}
The lower bound $h_K \|\boldsymbol  \chi (v)\|_{l^2} \lesssim \|v\|_{0,K}$ is trivial since $W_k(K)$ is a subspace of $V_{k,k}(K)$ and the d.o.f. in $V_{k,k}(K)$ contains additional moments in $\boldsymbol  \chi_{K}^{k}\backslash \boldsymbol  \chi_{K}^{k-2}$ compared with that of $W_k(K)$. To prove the upper bound, it suffices to bound these additional moments by the others.

By the definition of $W_k(K)$,
$$
(v, m)_{K} = (\Pi_k^{\nabla} v, m)_K, \quad \text{for all } m \in \mathbb  M_{k}(K)\backslash  \mathbb  M_{k-2}(K).
$$
Then by the  Cauchy-Schwarz inequality and the bound $\|m\|_{0,K}\lesssim h_K$, it suffices to bound $\|\Pi_k^{\nabla}v\|$. Using the d.o.f. of $v\in W_k(K)$, we can define another function $\tilde v\in V_k(K)$ s.t. $\boldsymbol  \chi (\tilde v) = \boldsymbol  \chi (v)$. Notice that the projection $\Pi_k^{\nabla}$ is uniquely determined by the d.o.f., therefore $$\Pi_k^{\nabla}v = \Pi_k^{\nabla}\tilde v.$$

Then by the $L^2$-stability of $\Pi_k^{\nabla}$ (cf. Corollary \ref{cor:PiK}) and the norm equivalence for $\tilde v\in V_{k}(K)$, we obtain
$$
\|\Pi_k^{\nabla}v\|_{0,K} = \|\Pi_k^{\nabla}\tilde v\|_{0,K}\lesssim \|\tilde v\|_{0,K}\lesssim h_K \|\boldsymbol  \chi (\tilde v)\|_{l^2} = h_K \|\boldsymbol  \chi (v)\|_{l^2}.
$$

Then for $\chi \in \boldsymbol  \chi_{K}^{k}\backslash \boldsymbol  \chi_{K}^{k-2}$, we can bound
$$
|\chi(v)| = |K|^{-1}|(v, m)_{K}|\lesssim h_K^{-1} \|\Pi_k^{\nabla}v\| \lesssim \|\boldsymbol  \chi (v)\|_{l^2}.
$$
The proof is thus completed.
\end{proof}

\subsection{Norm equivalence of VEM formulation}
With Theorem \ref{th:norm}, we can verify the following stability result.
\begin{theorem}[Norm equivalence for stabilization using $\Pi_k^{\nabla}$]\label{th:VEMPi}
For $u\in V_k(K)$ or $W_k(K)$, we have the following norm equivalence
$$
\|\nabla u\|_{0,K}^2 \eqsim \|\nabla \Pi_k^{\nabla} u\|_{0,K}^2 + \|\boldsymbol  \chi (u - \Pi_k^{\nabla} u)\|_{l^2}^2.
$$
\end{theorem}
\begin{proof}
By the definition of $\Pi_k^{\nabla}$, we have the orthogonality:
$$
\|\nabla u\|_{0,K}^2 = \|\nabla \Pi_k^{\nabla} u\|_{0,K}^2 + \|\nabla (u - \Pi_k^{\nabla} u)\|_{0,K}^2.
$$
Using the inverse inequality and norm equivalence for $L^2$-norm, we obtain
$$
 \|\nabla (u - \Pi_k^{\nabla} u)\|_{0,K}\lesssim h_K^{-1} \|u - \Pi_k^{\nabla} u\|_{0,K}\lesssim \|\boldsymbol  \chi (u - \Pi_k^{\nabla} u)\|_{l^2}.
$$
For $u\in W_k(K)\subset V_{k,k}(K)$, when we apply the norm equivalence for functions $V_{k,k}(K)$, additional moments in $\boldsymbol  \chi_{K}^{k}\backslash \boldsymbol  \chi_{K}^{k-2}$ should be involved. However, when applied to $u - \Pi_k^{\nabla} u$, by the definition of $W_k(K)$, these moments vanished and no need to include these moments.

To prove the lower bound, we shall apply the Poincar\'e-Friedrichs inequality (cf. Lemma \ref{lm:PF}) and the lower bound in the norm equivalence to get
$$
\|\boldsymbol  \chi (u - \Pi_k^{\nabla} u)\|_{l^2}\lesssim \| u - \Pi_k^{\nabla} u\|_{0,K}\lesssim  h_K\|\nabla (u - \Pi_k^{\nabla} u)\|_{0,K}.
$$
\end{proof}

Following \cite{Ahmad2013}, we introduce the $L^2$-projection $\Pi_k^0: W_k(K)\to \mathbb  P_k(K)$ and verify the stability of another stabilization using $\Pi_k^0$. For moments up to $k-2$, we can use the d.o.f. of VEM function $v$ in $W_k(K)$ and for higher moments, we use $\Pi_k^{\nabla}v$. That is: given $v\in W_k(K)$, define $\Pi_k^0 v \in \mathbb  P_k(K)$ such that
$$
\begin{cases}
(\Pi_k^0 v, m)_K = (v, m)_K, & \text{for all } m \in \mathbb  P_{k-2}(K),\\
(\Pi_k^0 v, m)_K = (\Pi_k^{\nabla} v, m)_K, & \text{for all }  m \in \mathbb  P_k(K) \backslash\mathbb  P_{k-2}(K).
\end{cases}
$$
Using the slice operator $I - \Pi_k^0$, the stabilization can be reduced to the d.o.f. on the boundary only.
\begin{corollary}[Norm equivalence for stabilization using $\Pi_k^0$]
For $u\in W_k(K)$, we have the following norm equivalence
$$
\|\nabla u\|_{0,K}^2 \eqsim \|\nabla \Pi_k^{\nabla} u\|_{0,K}^2 + \|\boldsymbol  \chi_{\partial K} (u - \Pi_k^0 u)\|_{l^2}^2.
$$
\end{corollary}
\begin{proof}
As both $\Pi_k^{\nabla}$ and $\Pi_k^0$ will preserve polynomial of degree $k$, we can write $(I - \Pi_k^0) u = (I - \Pi_k^0) (I - \Pi_k^{\nabla})u$ and $(I - \Pi_k^{\nabla}) u =  (I - \Pi_k^{\nabla})(I - \Pi_k^0)u$.

Using the stability of $\Pi_k^{\nabla}$ in $H^1$-semi-norm and the inverse inequality for VEM functions, we get
$$
\|\nabla (I - \Pi_k^{\nabla}) u\|_{0,K} = \|\nabla (I - \Pi_k^{\nabla}) (I - \Pi_k^0)u\|_{0,K} \leq \|\nabla (I - \Pi_k^0) u\|_{0,K}\lesssim h_k^{-1}\| (I - \Pi_k^0) u\|_{0,K}.
$$
Going backwards, we use the approximation property of the $L^2$-projection
$$
\| (I - \Pi_k^0) u\|_{0,K} = \|(I - \Pi_k^0) (I - \Pi_k^{\nabla})u\|_{0,K}\lesssim h_K\| \nabla (I - \Pi_k^{\nabla}) u\|_{0,K}.
$$
In summary, we have proved the norm equivalence
$$
h_K^{-1}\| (I - \Pi_k^0) u\|_{0,K} \eqsim \|\nabla (I - \Pi_k^{\nabla}) u\|_{0,K}.
$$

We then apply the norm equivalence Theorem \ref{th:norm} to function $(I - \Pi_k^0) u\in W_{k}(K)$. Note that the moment d.o.f. $\boldsymbol  \chi_K$ are vanished by the definition of $\Pi_k^0$ and thus only boundary d.o.f. $\boldsymbol  \chi_{\partial K}$ are presented in the stabilization term.
\end{proof}

\begin{remark}\rm
We can define a $L^2$-projection $\Pi_k^0$ to $\mathbb  P_k$ using moments d.o.f. of a VEM function in $V_{k,k}(K)$. Given a function $v\in V_{k,k}(K)$, if we denote by $v^0 = \Pi_k^0 v$ and $v^b = v|_{\partial K}$, then $(v^b, v^0)$ is a variant of the so-called weak function introduced in the weak Galerkin methods (cf. \cite{MuWangYe2015polyredu}). The stabilization term can be formulated as
$$
(\boldsymbol  \chi_{\partial K} (u^b - u^0), \boldsymbol  \chi_{\partial K} (v^b - v^0)).
$$
The approximate gradient $\nabla \Pi_k^{\nabla} v$ is indeed a variant of a weak gradient of the weak function $(v^b, v^0)$. It is also equivalent to a special version of HDG: the embedded discontinuous Galerkin method (cf. \cite{Cockburn2009a,Guzey2007}).
\end{remark}

\section{Interpolation Error Estimates}
In this section, we shall provide interpolation error estimates for several interpolations to the VEM spaces.
We introduce the following projection and interpolants of a function $v\in H^1(K)\cap C^0(\bar K)$:
\begin{itemize}
 \item $v_{\pi}\in \mathbb P_k(K)$: the $L^2$ projection of $v$ to the polynomial space;

 \item $v_c\in S_k(\mathcal T_K)$: the standard nodal interpolant to finite element space $S_k(\mathcal T_K)$ based on the auxiliary triangulation $\mathcal T_K$ of $K$;

 \item $v_I\in V_k(K)$ defined as the solution of the local problem
 $$
 \Delta v_I = \Delta u_{\pi} \text{ in } K, \quad v_I = v_c \text{ on } \partial K.
 $$

 \item  $I_Kv\in V_k(K)$ defined by d.o.f., i.e.,
 $$
I_K v = v_c \text{ on } \partial K, \quad (I_K v, p)_K = (v, p)_K, \; \forall p\in \mathbb P_{k-2}(K).
 $$

 \item $I_K^Wv\in W_k(K)$ defined by d.o.f., i.e.,
 $$
I_K^W v = v_c \text{ on } \partial K, \quad  (I_K^W v, p)_K = (v, p)_K, \; \forall p\in \mathbb P_{k-2}(K).
 $$
\end{itemize}

Error estimate of $v_{\pi}$ and $v_{c}$ are well known (cf.  \cite{Brenner.S;Scott.L2008}): for $w_K = v_c$ or $v_{\pi}$
\begin{equation}\label{vII}
\|v - w_K\|_{0,K} + h_{K}|v-w_K|_{1,K} \lesssim h_K^{k+1}\|v\|_{k+1,K}, \quad \forall v\in H^{k+1}(K).
\end{equation}
\begin{remark}\rm
Error estimate for $v_{\pi}$ is usually presented for a star shaped domain but can be generalized to a domain which is a union of star shaped sub-domains (cf. \cite{Dupont.T;Scott.R1980}). Under assumption {\bf A2}, the polygon $K$ satisfies the previous condition, so the estimate \eqref{vII} holds for $w_K=v_{\pi}$.
\end{remark}

The following error estimate can be found in \cite[Proposition 4.2]{Mora2015}. For the completeness, we present a shorter proof by comparing $v_I$ with $v_c$.
\begin{lemma}[Interpolation error estimate of $u_I$]\label{lm:uI}
The following optimal order error estimate holds:
\begin{equation}\label{vI}
\|v - v_I\|_{0,K} + h_K|v-v_I|_{1,K} \lesssim h_K^{k+1}\|v\|_{k+1,K}, \quad \forall v\in H^{k+1}(K).
\end{equation}
\end{lemma}
\begin{proof}
By the triangle inequality, it suffices to estimate the difference $v_I -v_c \in H_0^1(K)$. By the Poincar\'e inequality $\|v\|_{0,K} \lesssim h_K  \|\nabla v\|_{0,K}$  for $v\in H_0^1(K)$, it suffices to estimate the $H^1$-semi-norm.

 By the definition of $v_I$ and the fact $v_I -v_c \in H_0^1(K)$, we have
$$
(\nabla v_I, \nabla (v_I -v_c))_K = (\nabla v_{\pi}, \nabla (v_I -v_c))_K.
$$
Therefore
$$
\| \nabla (v_I -v_c) \|^2_{0,K} = (\nabla (v_{\pi} - v_c), \nabla (v_I -v_c))_K.
$$
By the Cauchy-Schwarz inequality and the triangle inequality, we then get
$$
\| \nabla (v_I -v_c) \|_{0,K}\leq \| \nabla (v_{\pi} -v_c) \|_{0,K}\leq \| \nabla (v -v_c) \|_{0,K} + \| \nabla (v -v_{\pi}) \|_{0,K}.
$$
The desired result then follows readily from error estimates for $v_c$ and $v_{\pi}$ together.
\end{proof}

Now we estimate $v - I_K v$ by comparing $I_K v$ with $v_I$.
\begin{theorem}[Interpolation error estimate of $I_Kv$]\label{th:IK}
For $v\in H^{k+1}(K)$, we have the optimal order error estimate in both $L^2$ and $H^1$-norm
$$ \|v - I_Kv\|_{0,K} + h_K |v- I_K v|_{1,K} \lesssim h_K^{k+1}\|v\|_{k+1,K}.$$
\end{theorem}
\begin{proof}
By the triangle inequality and error estimate on $v_I$ (cf. \eqref{vI}), it suffices to estimate $v_I - I_K v \in H_0^1(K)$ as follows:
\begin{align*}
(\nabla (v_I - I_K v), \nabla (v_I - I_K v))_K &= - (\Delta (v_I - I_K v), v_I - I_K v)_K\\
& = (\Delta (v_I - I_K v), v - v_I)_K\\
&\leq \|\Delta (v_I - I_K v)\|_{0,K}\|v - v_I \|_{0,K}\\
&\leq h_K^{-1}\|\nabla (v_I - I_K v)\|_{0,K} h_K^{k+1}\|v\|_{k+1,K}.
\end{align*}
The first step involves integration by parts and the fact $v_I - I_Kv\in H_0^1(K)$. The term $(\Delta (v_I - I_K v), v - I_Kv) = 0$ is due to $\Delta (v_I - I_K v)\in \mathbb P_{k-2}(K)$ and the moment preservation of the canonical interpolation. The last step uses the inverse inequality for $\Delta (v_I - I_K v)\in \mathbb P_{k-2}(K)$ (cf. \eqref{lm:invpoly}) and error estimate of $v - v_I$ (cf. \eqref{vI}).

Canceling one $\|\nabla (v_I - I_K v)\|_{0,K}$, we obtain the desired error estimate.
\end{proof}

Next, we present the interpolation error estimate of $ v - I_K^W v$ by comparing $I_K^Wv$ with $I_Kv$.

\begin{theorem}[Interpolation error estimate of $I_K^Wv$]\label{th:IK}
For $v\in H^{k+1}(K)$, we have the optimal order error estimate in both $L^2$ and $H^1$-norm
$$ \| v - I_K^W v\|_{0,K} + h_K \| \nabla (v - I_K^W v)\|_{0,K} \lesssim h_K^{k+1}\|v\|_{k+1,K}.$$
\end{theorem}
\begin{proof}
Again by the triangle inequality and the known error estimate for $v - I_Kv$, it suffices to estimate $I_K^W v- I_K v\in H_0^1(K)$.
   %
%
A crucial observation is that both interpolants, although in different VEM spaces, share the same d.o.f., i.e., $\boldsymbol  \chi(I_K^W v) = \boldsymbol  \chi(I_K v)$. Therefore $\Pi_k^{\nabla}I_K^W v =  \Pi_k^{\nabla} I_K v = \Pi_k^{\nabla} v.$

Using the norm equivalence (cf. Theorem \ref{th:VEMPi}), we have:
\begin{align*}
\|\nabla (I_K^W v- I_K v)\|_{0,K} &\leq \|\nabla (I-\Pi_k^{\nabla}) I_K^Wv\|_{0,K} + \|\nabla (I-\Pi_k^{\nabla}) I_Kv\|_{0,K} \\
&\lesssim  \|\boldsymbol  \chi (I - \Pi_k^{\nabla}) I_K^W v\|_{l^2} + \|\nabla (I-\Pi_k^{\nabla}) I_Kv\|_{0,K}\\
&= \|\boldsymbol  \chi (I - \Pi_k^{\nabla}) I_K v\|_{l^2} + \|\nabla (I-\Pi_k^{\nabla}) I_Kv\|_{0,K}\\
&\lesssim \|\nabla (I-\Pi_k^{\nabla}) I_Kv\|_{0,K}\\
&\lesssim \| \nabla (v - \Pi_k^{\nabla} v)\|_{0,K} + \| \nabla (v - I_K v) \|_{0,K}\\
&\lesssim \| \nabla (v - I_K v) \|_{0,K},
\end{align*}
as required.
\end{proof}

\begin{remark}\rm
Notice that we cannot apply the norm equivalence to $I_K^W v- I_K v$ directly since they are in different spaces.
Here we use the relations $\Pi_k^{\nabla}I_K^W v =  \Pi_k^{\nabla} I_K v = \Pi_k^{\nabla}v$ and $\boldsymbol  \chi(I_K^W v) = \boldsymbol  \chi(I_K v)$
as a bridge to switch the estimate for $I_K^Wv$ to that of $I_Kv$.
\end{remark}

\section{Conclusion and Future Work}
In this paper we have established the inverse inequality, norm equivalence between the norm of a finite element function and its degrees of freedom, and interpolation error estimates for several VEM spaces on a polygon which admits a virtual quasi-uniform triangulation, i.e., assumption {\bf A2}.

We note that {\bf A2} will rule out polygons with high aspect ratio. Equivalently the constant is not robust to the aspect ratio of $K$. For example, a rectangle $K$ with two sides $h_{\max}$ and $h_{\min}$. It can be decomposed into union of shape regular rectangles but the number will depend on the aspect ratio $h_{\max}/h_{\min}$. In numerical simulation, however, VEM is also robust to the aspect ratio of the elements. In a forthcoming paper, we will examine anisotropic error analysis of VEM based on certain maximum angle conditions.

We present our proofs in two dimensions but it is possible to extend the techniques to three dimensions. The outline is given as follows. Given a polyhedral region $K$, we need to assume {\bf A2} holds for each face $f\in \partial K$ and are able to prove results restricted to each face. Then we assume {\bf A2} holds for $K$ and prove results as for the 2-D case. It is our ongoing study to develop the details in this case.

\end{document}